\documentclass[12pt]{amsart}
\usepackage{amssymb}
\usepackage{t1enc}
\usepackage[latin2]{inputenc}
\usepackage{verbatim}
\usepackage{amsmath,amsfonts,amssymb,amsthm}
\usepackage[mathcal]{eucal}
\usepackage{enumerate}
\usepackage[centertags]{amsmath}
\usepackage{graphics}

\newtheorem{theorem}{Theorem}
\newtheorem{lemma}{Lemma}

\newtheorem{corollary}{Corollary}

\numberwithin{equation}{subsection}

\begin{document}
\author{N. Memić, I. Simon and G. Tephnadze}
\title[Strong convergence ]{Strong convergence of two--dimensional
Vilenkin-Fourier series}
\address{N. Memić, Department of Mathematics, University of Sarajevo, Zmaja
od Bosne 33-35, Sarajevo, Bosnia and Herzegovina.}
\email{nacima.o@gmail.com}
\address{I. Simon, Institute of Mathematics and Informatics, University of P%
\.{e}cs H-7624 P\.{e}cs, Ifjúság u. 6. Hungary.}
\email{simoni@gamma.ttk.pte.hu}
\address{G. Tephnadze, Department of Mathematics, Faculty of Exact and
Natural Sciences, Tbilisi State University, Chavchavadze str. 1, Tbilisi
0128, Georgia and Department of Engineering Sciences and Mathematics, Lule%
\aa {} University of Technology, SE-971 87, Lule\aa {}, Sweden.}
\email{giorgitephnadze@gmail.com}
\date{}
\maketitle

\begin{abstract}
We prove that certain means of the quadratical partial sums of the
two-dimensional Vilenkin-Fourier series are uniformly bounded operators from
the Hardy space $H_{p}$ to the space $L_{p}$ for $0<p\leq 1.$ We also prove
that the sequence in the denominator cannot be improved.
\end{abstract}

\textbf{2010 Mathematics Subject Classification.} 42C10.

\textbf{Key words and phrases:} Vilenkin systems, Strong convergence,
martingale Hardy space.

\section{INTRODUCTION}

The definitions and notations used in this introduction can be found in our
next Section. It is known \cite[p. 125]{G-E-S} that Vilenkin systems are not
Schauder bases in $L_{1}\left( G_{m}\right) $. Moreover, (see \cite{AVD} and
\cite{S-W-S}) there exists a function in the Hardy space $H_{1}\left(
G_{m}\right) $, the partial sums of which are not bounded in $L_{1}\left(
G_{m}\right) .$ However, in Gát \cite{gat1} (see also \cite{b}) the
following strong convergence result was obtained for all $f\in H_{1}:$%
\begin{equation*}
\underset{n\rightarrow \infty }{\lim }\frac{1}{\log n}\overset{n}{\underset{%
k=1}{\sum }}\frac{\left\Vert S_{k}f-f\right\Vert _{1}}{k}=0,
\end{equation*}%
where Vilenkin system is bounded and $S_{k}f$ denotes the $k$-th partial sum
of the Vilenkin-Fourier series of $f$ \ (For the trigonometric analogue see
Smith \cite{sm}, for Walsh system see Simon \cite{Si}).

Simon \cite{si1} (see also \cite{tep4}) proved that\textbf{\ }there is an
absolute constant $c_{p},$ depending only on $p,$ such that
\begin{equation}
\overset{\infty }{\underset{k=1}{\sum }}\frac{\left\Vert S_{k}f\right\Vert
_{p}^{p}}{k^{2-p}}\leq c_{p}\left\Vert f\right\Vert _{H_{p}}^{p},
\label{1cc}
\end{equation}%
for all $f\in H_{p}\left( G_{m}\right) ,$ where $0<p<1.$

In \cite{tep5} it was proved that the sequence $\left\{ 1/k^{2-p}:k\in
\mathbb{N}\right\} $ in inequality (\ref{1cc}) can not be improved.

For the two-dimensional Walsh-Fourier series Goginava and Gogoladze \cite{gg}
for $p=1$ and Tephnadze \cite{tep1} for $0<p<1$ proved that there exists an
absolute constant $c_{p}$, depending only on $p,$ such that
\begin{equation}
\sum\limits_{n=1}^{\infty }\frac{\left\Vert S_{n,n}f\right\Vert _{p}}{%
n^{3-2p}\log ^{2\left[ p\right] }\left( n+1\right) }\leq c_{p}\left\Vert
f\right\Vert _{H_{p}^{\square }},  \label{2cc}
\end{equation}%
for all $f\in H_{p}^{\square }\left( G_{m}^{2}\right) $, where $\left[ p%
\right] $ denotes integer part of $p.$

Moreover, in \cite{tep1} and \cite{tep2} there were proved that sequence $%
\left\{ 1/k^{3-2p}\log ^{2\left[ p\right] }\left( k+1\right) :k\in \mathbb{N}%
\right\} $ in inequality (\ref{2cc}) is important.

Convergence of quadratical partial sums of two-dimensional Walsh-Fourier
series was investigated in detail by Gát, Goginava, Nagy \cite{GGN},
Goginava, Tkebuchava \cite{GGT}, Goginava \cite{go}, \ Gogoladze \cite{Go},
Weisz \cite{Webook2}. Some strong convergence theorem for the
two-dimensional case can be found in \cite{B-T, B-T-T}, \cite{gog8, gog9,
gg2}, \cite{tep3,nt2,nt3} and \cite{We}.

The main aim of this paper is to generalize inequality (\ref{2cc}) (see
Theorem \ref{theorem1}) for bounded Vilenkin system. We also prove that the
sequence $\left\{ 1/k^{3-2p}\log ^{2\left[ p\right] }\left( k+1\right) :k\in
\mathbb{N}\right\} $ in inequality (\ref{2cc}) can not be improved (see
Theorem \ref{theorem2}). We note that the counterexample which shows
sharpness of previous results is new even for Walsh case.

This paper is organized as follows: in order not to disturb our discussions
later on some definitions and notations are presented in Section 2. The main
results and some of its consequences can be found in Section 3. For the
proofs of the main results we need some auxiliary results of independent
interest. These results are presented in Section 4. The detailed proofs
are given in Section 5.

\section{DEFINITIONS AND NOTATION}

Denote by $\mathbb{N}_{+}$ the set of positive integers, $\mathbb{N}:=%
\mathbb{N}_{+}\cup \{0\}.$ Let $m:=(m_{0,}$ $m_{1},...)$ be a sequence of
positive integers not less than 2. Denote by
\begin{equation*}
Z_{m_{k}}:=\{0,1,\ldots ,m_{k}-1\}
\end{equation*}%
the additive group of integers modulo $m_{k}$.

Define the group $G_{m}$ as the complete direct product of the groups $%
Z_{m_{i}}$ with the product of the discrete topologies of $Z_{m_{j}}`$s.

\textbf{In this paper we discuss bounded Vilenkin groups,\ i.e. the case when%
}
\begin{equation*}
\sup_{n}m_{n}<\infty .
\end{equation*}

The direct product $\mu $ of the measures
\begin{equation*}
\mu _{k}\left( \{j\}\right) :=1/m_{k},\ (j\in Z_{m_{k}})
\end{equation*}%
is the Haar measure on $G_{m_{\text{ }}}$ with $\mu \left( G_{m}\right) =1.$

The elements of $G_{m}$ are represented by sequences
\begin{equation*}
x:=\left( x_{0},x_{1},\ldots ,x_{j},\ldots \right) ,\ \left( x_{j}\in
Z_{m_{j}}\right) .
\end{equation*}

It is easy to give a base of neighbourhoods of $G_{m}:$

\begin{equation*}
\text{ }I_{0}\left( x\right) :=G_{m},\text{ \ }I_{n}(x):=\{y\in G_{m}\mid
y_{0}=x_{0},\ldots ,y_{n-1}=x_{n-1}\}\,\,\left( x\in G_{m},\text{ }n\in
\mathbb{N}\right) .
\end{equation*}

Denote $I_{n}:=I_{n}\left( 0\right) ,$ for $n\in \mathbb{N}_{+}$ and
\begin{equation*}
e_{n}:=\left( 0,\ldots ,0,x_{n}=1,0,\ldots \right) \in G_{m},\text{ \ }%
\left( n\in \mathbb{N}\right) .
\end{equation*}

It is evident that
\begin{equation}
\overline{I_{N}}=\overset{N-1}{\underset{s=0}{\bigcup }}I_{s}\backslash
I_{s+1}.  \label{add}
\end{equation}

\bigskip If we define the so-called generalized number system based on $m$
in the following way :
\begin{equation*}
M_{0}:=1,\ M_{k+1}:=m_{k}M_{k}\,\,\,\ \ (k\in \mathbb{N}),
\end{equation*}%
then every $n\in \mathbb{N}$ can be uniquely expressed as
\begin{equation*}
n=\sum_{j=0}^{\infty }n_{j}M_{j},\text{ \ where \ \ }n_{j}\in Z_{m_{j}}\text{
\ }(j\in \mathbb{N})
\end{equation*}%
and only a finite number of $n_{j}`$s differ from zero. Let $|n|$ denote the
largest integer $j$ for which $n_{j}\neq 0$.

Denote by $\mathbb{N}_{n_{0}}$ the subset of positive integers $\mathbb{N}%
_{+},$ for which $n_{0}=1.$ Then for every $n\in \mathbb{N}_{n_{0}},$ $%
M_{k}<n<$ $M_{k+1}$ can be written as
\begin{equation*}
n=1+\sum_{j=1}^{k}n_{j}M_{j},
\end{equation*}%
where%
\begin{equation*}
n_{j}\in \left\{ 0,\ldots {},m_{j}-1\right\} ,~(j=1,...,k-1)\text{ and \ }%
n_{k}\in \left\{ 1,\ldots {},m_{k}-1\right\} .
\end{equation*}

By simple calculation we get that
\begin{equation}
\underset{\left\{ n:M_{k}\leq n\leq M_{k+1},\text{ }n\in \mathbb{N}%
_{n_{0}}\right\} }{\sum }1=\frac{M_{k+1}-M_{k}}{m_{0}}\geq cM_{k},
\label{1a}
\end{equation}%
where $c$ is an absolute constant.

For any natural number $n=\sum_{j=1}^{\infty }n_{j}M_{j},$ we define the
functions $v\left( n\right) $ and $v^{\ast }\left( n\right) $ by

\begin{equation}
v\left( n\right) =\sum_{j=1}^{\infty }\left\vert \delta _{j+1}-\delta
_{j}\right\vert +\delta _{0},\text{ \ }v^{\ast }\left( n\right)
=\sum_{j=1}^{\infty }\delta _{j}^{\ast },  \label{100}
\end{equation}%
where
\begin{equation*}
\delta _{j}=signn_{j}=sign\left( \ominus n_{j}\right) ,\text{ \ \ \ \ }%
\delta _{j}^{\ast }=\left\vert \ominus n_{j}-1\right\vert \delta _{j},
\end{equation*}%
and $\ominus $\ is the inverse operation for
\begin{equation*}
a_{k}\oplus b_{k}:=(a_{k}+b_{k})\text{mod}m_{k}.
\end{equation*}

The norm (or quasi-norm) of the Lebesgue spaces $L_{p}(G_{m})$ $\left(
0<p<\infty \right) $ are defined in the usual way.

Next, we introduce on $G_{m}$ an orthonormal system which is called the
Vilenkin system.

At first, we define the complex-valued function $r_{k}\left( x\right)
:G_{m}\rightarrow \mathbb{C},$ the generalized Rademacher functions, by
\begin{equation*}
r_{k}\left( x\right) :=\exp \left( 2\pi ix_{k}/m_{k}\right) ,\text{ }\left(
i^{2}=-1,x\in G_{m},\text{ }k\in \mathbb{N}\right) .
\end{equation*}

Now, define the Vilenkin system $\,\,\,\psi :=(\psi _{n}:n\in\mathbb{N})$ on
$G_{m}$ as:
\begin{equation*}
\psi _{n}(x):=\prod\limits_{k=0}^{\infty }r_{k}^{n_{k}}\left( x\right)
,\,\,\ \ \,\left( n\in\mathbb{N}\right).
\end{equation*}

Specifically, we call this system the Walsh-Paley system, when $m\equiv 2.$

The Vilenkin system is orthonormal and complete in $L_{2}\left( G_{m}\right)
$ (see \cite{AVD} and \cite{Vi}).

Now, we introduce analogues of the usual definitions in Fourier-analysis. If
$f\in L_{1}\left( G_{m}\right) $ we can define the Fourier coefficients, the
partial sums of the Fourier series, the Dirichlet kernels with respect to
the Vilenkin systems in the usual way:

\begin{eqnarray*}
\widehat{f}\left( n\right) &:&=\int_{G_{m}}f\overline{\psi }_{n}d\mu \,,%
\text{ \ \ \ }\left( n\in \mathbb{N}\right) \\
\text{ \ \ }S_{n}f &:&=\sum_{k=0}^{n-1}\widehat{f}\left( k\right) \psi _{k},%
\text{ \ \ \ }\left( n\in \mathbb{N}_{+}\right) \\
D_{n} &:&=\sum_{k=0}^{n-1}\psi _{k\text{ }},\text{\ \ \ \ \ \ \ \ \ \ \ }%
\left( n\in \mathbb{N}_{+}\right)
\end{eqnarray*}%
respectively.

Recall that
\begin{equation}
D_{M_{n}}\left( x\right) =\left\{
\begin{array}{ll}
M_{n}, & \text{if\thinspace \thinspace \thinspace }x\in I_{n}, \\
0, & \text{if}\,\,x\notin I_{n}.%
\end{array}%
\right.  \label{1dn}
\end{equation}

It is also known that (see \cite{AVD} and \cite{blahota})
\begin{equation}
D_{sM_{n}}=D_{M_{n}}\sum_{k=0}^{s-1}\psi
_{kM_{n}}=D_{M_{n}}\sum_{k=0}^{s-1}r_{n}^{k},  \label{2dn}
\end{equation}%
and
\begin{equation}
D_{n}=\psi _{n}\left( \sum_{j=0}^{\infty
}D_{M_{j}}\sum_{u=m_{j}-n_{j}}^{m_{j}-1}r_{j}^{u}\right) .  \label{3dn}
\end{equation}

Denote by $L_{n}$ the $n$-th Lebesgue constant%
\begin{equation*}
L_{n}:=\left\Vert D_{n}\right\Vert _{1}.
\end{equation*}

The norm (or quasi-norm) of the space $L_{p}(G_{m}^{2})$ is defined by \qquad

\begin{equation*}
\left\Vert f\right\Vert _{p}:=\left( \int_{G_{m}^{2}}\left\vert f\right\vert
^{p}d\mu \times d\mu \right) ^{1/p}\qquad \left( 0<p<\infty \right) .
\end{equation*}

The space $weak-L_{p}\left( G_{m}^{2}\right) $ consists of all measurable
functions $f,$ for which

\begin{equation*}
\left\Vert f\right\Vert _{weak-L_{p}}:=\underset{\lambda >0}{\sup }\lambda
\mu \left( f>\lambda \right) ^{1/p}<+\infty .
\end{equation*}

The rectangular partial sums of the 2-dimensional Vilenkin-Fourier series of
function $f\in L_{2}\left( G_{m}^{2}\right) $ are defined as follows:

\begin{equation*}
S_{M,N}f\left( x,y\right) :=\sum\limits_{i=0}^{M-1}\sum\limits_{j=0}^{N-1}%
\widehat{f}\left( i,j\right) \psi _{i}\left( x\right) \psi _{j}\left(
y\right) ,
\end{equation*}%
where the numbers
\begin{equation*}
\widehat{f}\left( i,j\right) =\int\limits_{G_{m}^{2}}f\left( x,y\right) \bar{%
\psi}_{i}\left( x\right) \bar{\psi}_{j}\left( y\right) d\mu \left( x,y\right)
\end{equation*}%
is said to be the $\left( i,j\right) $-th Vilenkin-Fourier coefficient of
the function \thinspace $f.$

Denote
\begin{equation*}
S_{M}^{\left( 1\right) }f\left( x,y\right) :=\int\limits_{G_{m}}f\left(
s,y\right) D_{M}\left( x-s\right) d\mu \left( s\right)
\end{equation*}%
and
\begin{equation*}
S_{N}^{\left( 2\right) }f\left( x,y\right) :=\int\limits_{G_{m}}f\left(
x,t\right) D_{N}\left( y-t\right) d\mu \left( t\right) .
\end{equation*}

The $\sigma $-algebra generated by the intervals $\left\{ I_{n}\left(
x\right) :x\in G_{m}\right\} $ will be denoted by $\digamma _{n}\left( n\in
\mathbb{N}\right) .$ Denote by $f=\left( f^{\left( n\right) },n\in \mathbb{N}%
\right) $ a martingale with respect to $\digamma _{n}\left( n\in \mathbb{N}%
\right) .$ (for details see e.g. \cite{Webook2}).

The one-dimensional Hardy space $H_{p}(G_{m})$ $\left( 0<p<\infty \right) $
consists of all martingales, for which

\begin{equation*}
\left\Vert f\right\Vert _{H_{p}}:=\left\Vert \sup_{n\in \mathbb{N}%
}\left\vert f_{n}\right\vert \right\Vert _{p}<\infty .
\end{equation*}

The $\sigma $-algebra generated by the 2-dimensional $I_{n}\left( x\right)
\times I_{n}\left( y\right) $ square of measure $M_{n}^{-1}\times M_{n}^{-1}$
will be denoted by $\digamma _{n,n}\left( n\in \mathbb{N}\right) .$ Denote
by $f=\left( f_{n,n}\text{ }n\in \mathbb{N}\right) $ one-parameter
martingale with respect to $\digamma _{n,n}\left( n\in \mathbb{N}\right) .$%
(for details see e.g. \cite{Webook1}).

The maximal function of a martingale $f$ \ is defined by

\begin{equation*}
f^{\ast }=\sup_{n\in \mathbb{N}}\left\vert f_{n,n}\right\vert .
\end{equation*}

Let $f\in L_{1}\left( G_{m}^{2}\right) $. Then the maximal function is given
by
\begin{equation*}
f^{\ast }\left( x,y\right) =\sup\limits_{n\in \mathbb{N}}\frac{1}{\mu \left(
I_{n}(x)\times I_{n}(y)\right) }\left\vert \int\limits_{I_{n}(x)\times
I_{n}(y)}f\left( s,t\right) d\mu \left( s,t\right) \right\vert ,\,\,
\end{equation*}%
where $\left( x,y\right) \in G_{m}^{2}.$

The two-dimensional Hardy space $H_{p}^{\square }(G_{m}^{2})$ $\left(
0<p<\infty \right) $ consists of all functions for which

\begin{equation*}
\left\Vert f\right\Vert _{H_{p}^{\square }}:=\left\Vert f^{\ast }\right\Vert
_{p}<\infty .
\end{equation*}

If $f\in L_{1}\left( G_{m}^{2}\right) ,$ then it is easy to show that the
sequence $\left( S_{M_{n},M_{n}}\left( f\right) :n\in \mathbb{N}\right) $ is
a martingale. If $f=\left( f_{n,n},n\in \mathbb{N}\right) $ is a martingale,
then the Vilenkin-Fourier coefficients must be defined in a slightly
different manner: $\qquad \qquad $
\begin{equation*}
\widehat{f}\left( i,j\right) :=\lim_{k\rightarrow \infty
}\int_{G_{m}^{2}}f_{k,k}\left( x,y\right) \overline{\psi }_{i}\left(
x\right) \overline{\psi }_{j}\left( y\right) d\mu \left( x,y\right) .
\end{equation*}

The Vilenkin-Fourier coefficients of $f\in L_{1}\left( G_{m}^{2}\right) $
are the same as those of the martingale $\left( S_{M_{n},M_{n}}f:n\in
\mathbb{N}\right) $ obtained from $f$ .

A bounded measurable function $a$ is a $p$-atom, if there exists a
\thinspace two-dimensional cube $I^{2}=I\times I\mathbf{,}$ such that
\begin{equation*}
\int_{I^{2}}ad\mu =0,\text{ \ \ \ \ \ }\left\Vert a\right\Vert _{\infty
}\leq \mu (I^{2})^{-1/p},\text{ \ \ \ \ \ supp}\left( a\right) \subset I^{2}.
\end{equation*}

\section{\textbf{MAIN RESULTS AND SOME OF THEIR CONSEQUENCES}}

Our first main result reads:

\begin{theorem}
\label{theorem1}Let $0<p\leq 1$ and $f\in H_{p}^{\square }\left(
G_{m}^{2}\right) $. Then
\begin{equation*}
\sum\limits_{n=1}^{\infty }\frac{\left\Vert S_{n,n}f\right\Vert _{p}^{p}}{%
n^{3-2p}\log ^{2\left[ p\right] }\left( n+1\right) }\leq c_{p}\left\Vert
f\right\Vert _{H_{p}^{\square }}^{p},
\end{equation*}%
where $\left[ p\right] $ denotes integer part of $p.$
\end{theorem}

By using Theorem \ref{theorem1} we obtain the following:

\begin{corollary}
\label{corollary1}Let $0<p\leq 1$ and $f\in H_{p}^{\square }\left(
G_{m}^{2}\right) $. Then
\begin{equation*}
\sum\limits_{n=1}^{\infty }\frac{\left\Vert S_{n,n}f\right\Vert
_{H_{p}^{\square }}^{p}}{n^{3-2p}\log ^{2\left[ p\right] }\left( n+1\right) }%
\leq c_{p}\left\Vert f\right\Vert _{H_{p}^{\square }}^{p}.
\end{equation*}
\end{corollary}

Also this result is sharp in the following important sense:

\begin{theorem}
\label{theorem2}a) Let $0<p<1$ and $\Phi :\mathbb{N}\rightarrow \lbrack
1,\infty )$ be any non-decreasing function, satisfying the condition $%
\underset{n\rightarrow \infty }{\lim }\Phi \left( n\right) =+\infty .$ Then
there exists a martingale $f\in H_{p}^{\square }\left( G_{m}^{2}\right) $
such that%
\begin{equation*}
\underset{n=1}{\overset{\infty }{\sum }}\frac{\left\Vert S_{n,n}f\right\Vert
_{weak-L_{p}}^{p}\Phi \left( n\right) }{n^{3-2p}}=\infty .
\end{equation*}%
b) Let $\Phi :\mathbb{N}\rightarrow \lbrack 1,\infty )$ be any
non-decreasing function, satisfying the condition $\underset{n\rightarrow
\infty }{\lim }\Phi \left( n\right) =+\infty .$ Then there exists a
martingale $f\in H_{1}^{\square }\left( G_{m}^{2}\right) $ such that%
\begin{equation*}
\underset{n=1}{\overset{\infty }{\sum }}\frac{\left\Vert S_{n,n}f\right\Vert
_{1}\Phi \left( n\right) }{n\log ^{2}\left( n+1\right) }=\infty .
\end{equation*}
\end{theorem}

We also apply Theorem \ref{theorem1} to obtain that the following is true:

\begin{corollary}
\label{corollary2}Let $0<p\leq 1$ and $f\in H_{p}^{\square }\left(
G_{m}^{2}\right) $. Then
\begin{equation*}
\sum\limits_{n=1}^{\infty }\frac{\left\vert S_{n,n}f\left( x,y\right)
\right\vert ^{p}}{n^{3-2p}\log ^{2\left[ p\right] }\left( n+1\right) }%
<\infty ,\text{ \ \ a.e. \ \ }\left( x,y\right) \in G_{m}^{2}.
\end{equation*}
\end{corollary}

\begin{center}
\bigskip
\end{center}

\section{\textbf{AUXILIARY RESULTS}}

\bigskip In order to prove our main results we need the following lemma of
Weisz (for details see e.g. \cite{Webook1})

\begin{lemma}
\label{lemma1} A martingale $f\in L_{p}\left( G_{m}^{2}\right) $ is in $%
H_{p}^{\square }\left( G_{m}^{2}\right) \left( 0<p\leq 1\right) $ if and
only if there exist a sequence $\left( a_{k},k\in \mathbb{N}\right) $ of
p-atoms and a sequence $\left( \mu _{k},k\in \mathbb{N}\right) $ of real
numbers such that
\begin{equation}
\qquad \sum_{k=0}^{\infty }\mu _{k}S_{M_{n},M_{n}}a_{k}=f_{n,n},\text{ \ \ \
\ \ a.e.,}  \label{a1}
\end{equation}%
and%
\begin{equation*}
\qquad \sum_{k=0}^{\infty }\left\vert \mu _{k}\right\vert ^{p}<\infty ,
\end{equation*}%
Moreover,
\begin{equation*}
\left\Vert f\right\Vert _{H_{p}}\backsim \inf \left( \sum_{k=0}^{\infty
}\left\vert \mu _{k}\right\vert ^{p}\right) ^{1/p},
\end{equation*}%
where the infimum is taken over all decompositions of $f$ of the form (\ref%
{a1}).
\end{lemma}

We also state a new lemma, which we need for the proofs of our main results
but which is also of independent interest:

\begin{lemma}
\label{lemma2}Let $n\in \mathbb{N}$. Then
\begin{equation*}
\frac{1}{nM_{n}}\underset{k=1}{\overset{M_{n}-1}{\sum }}v\left( k\right)
\geq \frac{2}{\lambda ^{2}},
\end{equation*}%
where $\lambda =\sup_{n\in
\mathbb{N}
}m_{n}.$
\end{lemma}

\begin{proof}[\textbf{Proof} \textbf{of} \textbf{Lemma \protect\ref{lemma2}.}%
]
Let $M_{k-1}\leq n<M_{k}.$ Then $n=n_{k-1}M_{k-1}+n^{\left( 1\right) },$
where $n^{\left( 1\right) }<M_{k-1}.$ It is easy to show that

\begin{equation*}
\underset{n=M_{k-1}}{\overset{M_{k}-1}{\sum }}v\left( n\right) =\underset{r=1%
}{\overset{m_{k-1}-1}{\sum }}\underset{n=rM_{k-1}}{\overset{\left(
r+1\right) M_{k-1}-1}{\sum }}v\left( n\right) =\underset{r=1}{\overset{%
m_{k-1}-1}{\sum }}\underset{n=0}{\overset{M_{k-1}-1}{\sum }}v\left(
n+rM_{k-1}\right)
\end{equation*}%
\begin{equation*}
=\underset{r=1}{\overset{m_{k-1}-1}{\sum }}\underset{n=0}{\overset{M_{k-2}-1}%
{\sum }}v\left( n+rM_{k-1}\right) +\underset{r=1}{\overset{m_{k-1}-1}{\sum }}%
\underset{n=M_{k-2}}{\overset{M_{k-1}-1}{\sum }}v\left( n+rM_{k-1}\right)
\end{equation*}%
\begin{equation*}
=\underset{r=1}{\overset{m_{k-1}-1}{\sum }}\underset{n=0}{\overset{M_{k-2}-1}%
{\sum }}v\left( n+rM_{k-1}\right) +\underset{r=1}{\overset{m_{k-1}-1}{\sum }}%
\underset{n=M_{k-2}}{\overset{M_{k-1}-1}{\sum }}v\left( n+rM_{k-1}\right)
\end{equation*}%
\begin{equation*}
=\underset{r=1}{\overset{m_{k-1}-1}{\sum }}\underset{n=0}{\overset{M_{k-2}-1}%
{\sum }}\left( v\left( n\right) +2\right) +\underset{r=1}{\overset{m_{k-1}-1}%
{\sum }}\underset{n=M_{k-2}}{\overset{M_{k-1}-1}{\sum }}v\left( n\right)
\end{equation*}%
\begin{equation*}
\left( m_{k-1}-1\right) \left( \underset{n=0}{\overset{M_{k-2}-1}{\sum }}%
\left( v\left( n\right) +2\right) +\underset{n=M_{k-2}}{\overset{M_{k-1}-1}{%
\sum }}v\left( n\right) \right)
\end{equation*}%
\begin{equation*}
=\left( m_{k-1}-1\right) \left( \underset{n=0}{\overset{M_{k-1}-1}{\sum }}%
v\left( n\right) +2M_{k-2}\right) .
\end{equation*}%
This gives that%
\begin{equation*}
\underset{n=0}{\overset{M_{k}-1}{\sum }}v\left( n\right) =\underset{n=0}{%
\overset{M_{k-1}-1}{\sum }}v\left( n\right) +\underset{n=M_{k-1}}{\overset{%
M_{k}-1}{\sum }}v\left( n\right)
\end{equation*}%
\begin{equation*}
=\underset{n=0}{\overset{M_{k-1}-1}{\sum }}v\left( n\right) +\left(
m_{k-1}-1\right) \left( \underset{n=0}{\overset{M_{k-1}-1}{\sum }}v\left(
n\right) +2M_{k-2}\right)
\end{equation*}%
\begin{equation*}
=m_{k-1}\underset{n=0}{\overset{M_{k-1}-1}{\sum }}v\left( n\right) +2M_{k-2}
\end{equation*}

Let
\begin{equation*}
T\left( M_{k}\right) =\sum_{n=0}^{M_{k}-1}v\left( n\right) .
\end{equation*}%
Then
\begin{equation*}
T\left( M_{k}\right) =m_{k-1}T\left( M_{k-1}\right) +2M_{k-2}.
\end{equation*}%
It is easy to see that this is valid for all $k\in \mathbb{N}.$ If we set $\
$%
\begin{equation*}
A\left( k\right) :=T\left( M_{k}\right) /M_{k}
\end{equation*}%
then%
\begin{equation*}
A\left( k\right) \geq A\left( k-1\right) +2/\lambda ^{2},
\end{equation*}%
\begin{equation*}
\text{ \ \ }A\left( k-1\right) \geq A\left( k-2\right) +2/\lambda ^{2},
\end{equation*}%
\begin{equation*}
\cdots \cdots \cdots
\end{equation*}%
\begin{equation*}
A\left( 1\right) \geq A\left( 0\right) +2/\lambda ^{2}.
\end{equation*}

Summing up these inequalities we can write that%
\begin{equation*}
A\left( k\right) \geq 2k/\lambda ^{2}\text{ \ \ }
\end{equation*}%
and%
\begin{equation*}
T\left( M_{k}\right) \geq 2kM_{k}/\lambda ^{2},
\end{equation*}%
which completes the proof of Lemma \ref{lemma2}.
\end{proof}

\section{PROOFS OF MAIN RESULTS}

\begin{proof}[\textbf{Proof} \textbf{of} \textbf{Theorem \protect\ref%
{theorem1}.}]
According to Lemma \ref{lemma1} we only have to prove that%
\begin{equation}
\sum\limits_{n=1}^{\infty }\frac{\left\Vert S_{n,n}a\right\Vert _{p}^{p}}{%
n^{3-2p}\log ^{2\left[ p\right] }\left( n+1\right) }\leq c_{p}<\infty ,
\label{main}
\end{equation}%
for every $p$-atom $a$.

Let $\left( x,y\right) \in \overline{I}_{N}^{2}:=\overline{I}_{N}\times
\overline{I}_{N}$. Let $a$ be an arbitrary $p$-atom with support $%
I_{N}\left( z^{\prime }\right) \times I_{N}\left( z^{\prime \prime }\right) $
and $\mu \left( I_{N}\left( z^{\prime }\right) \right) =\mu \left(
I_{N}\left( z^{\prime \prime }\right) \right) =M_{N}^{-1}$. We can suppose
that $z^{\prime }=z^{\prime \prime }=0.$ In this case $D_{M_{i}}\left(
x-s\right) 1_{I_{N}}\left( s\right) =0$ and $D_{M_{i}}\left( y-t\right)
1_{I_{N}}\left( t\right) =0$ for $i\geq N$. Recall that $\psi _{j}\left(
x-t\right) =\psi _{j}\left( x\right) $ for $t\in I_{N}$ and $j<N$.
Consequently, from (\ref{1dn}) and (\ref{3dn}) we obtain that
\begin{equation*}
S_{n,n}a\left( x,y\right) =\int\limits_{G_{m}^{2}}a\left( s,t\right)
D_{n}\left( x-s\right) D_{n}\left( y-t\right) d\mu \left( s,t\right)
\end{equation*}%
\begin{equation*}
=\int\limits_{I_{N}^{2}}a\left( s,t\right) D_{n}\left( x-s\right)
D_{n}\left( y-t\right) d\mu \left( s,t\right)
\end{equation*}%
\begin{equation*}
=\int\limits_{I_{N}^{2}}a\left( s,t\right) \psi _{n}\left( x-s+y-t\right)
\times
\end{equation*}%
\begin{equation*}
\times \left( \sum_{j=0}^{N-1}D_{M_{j}}\left( x-s\right)
\sum_{u=m_{j}-n_{j}}^{m_{j}-1}r_{j}^{u}\left( x-s\right) \right) \left(
\sum_{j=0}^{N-1}D_{M_{j}}\left( y-t\right)
\sum_{u=m_{j}-n_{j}}^{m_{j}-1}r_{j}^{u}\left( y-t\right) \right) d\mu \left(
s,t\right)
\end{equation*}%
\begin{equation*}
=\psi _{n}\left( x+y\right) \left( \sum_{j=0}^{N-1}D_{M_{j}}\left( x\right)
\sum_{u=m_{j}-n_{j}}^{m_{j}-1}r_{j}^{u}\left( x\right) \right) \left(
\sum_{j=0}^{N-1}D_{M_{j}}\left( y\right)
\sum_{u=m_{j}-n_{j}}^{m_{j}-1}r_{j}^{u}\left( y\right) \right) \times
\end{equation*}%
\begin{equation*}
\times \int\limits_{I_{N}^{2}}a\left( s,t\right) \overline{\psi }_{n}\left(
s+t\right) d\mu \left( s,t\right)
\end{equation*}%
\begin{equation*}
=\psi _{n}\left( x+y\right) \left( \sum_{j=0}^{N-1}D_{M_{j}}\left( x\right)
\sum_{u=m_{j}-n_{j}}^{m_{j}-1}r_{j}^{u}\left( x\right) \right) \left(
\sum_{j=0}^{N-1}D_{M_{j}}\left( y\right)
\sum_{u=m_{j}-n_{j}}^{m_{j}-1}r_{j}^{u}\left( y\right) \right)
\end{equation*}%
\begin{equation*}
\times \int\limits_{I_{N}}\left( \int\limits_{I_{N}}a\left( \tau -t,t\right)
d\mu \left( t\right) \right) \overline{\psi }_{n}\left( \tau \right) d\mu
\left( \tau \right)
\end{equation*}%
\begin{equation*}
=\psi _{n}\left( x+y\right) \left( \sum_{j=0}^{N-1}D_{M_{j}}\left( x\right)
\sum_{u=m_{j}-n_{j}}^{m_{j}-1}r_{j}^{u}\left( x\right) \right) \left(
\sum_{j=0}^{N-1}D_{M_{j}}\left( y\right)
\sum_{u=m_{j}-n_{j}}^{m_{j}-1}r_{j}^{u}\left( y\right) \right) \widehat{\Phi
}\left( n\right) ,
\end{equation*}

where
\begin{equation*}
\Phi \left( \tau \right) =\int\limits_{I_{N}}a\left( \tau -t,t\right) d\mu
\left( t\right) .
\end{equation*}

Let $x\in I_{s}\backslash I_{s+1}.$ According to (\ref{3dn}) we get that
\begin{equation}
\left\vert \sum_{j=0}^{N-1}D_{M_{j}}\left( x\right)
\sum_{u=m_{j}-n_{j}}^{m_{j}-1}r_{j}^{u}\left( x\right) \right\vert \leq
\left\vert \sum_{j=0}^{s}D_{M_{j}}\left( x\right)
\sum_{u=m_{j}-n_{j}}^{m_{j}-1}r_{j}^{u}\left( x\right) \right\vert \leq
cM_{s+1}.  \label{91}
\end{equation}

Combining (\ref{add}) and (\ref{91}) we have that
\begin{equation}
\int_{\overline{I}_{N}}\left\vert \sum_{j=0}^{N-1}D_{M_{j}}\left( x\right)
\sum_{u=m_{j}-n_{j}}^{m_{j}-1}r_{j}^{u}\left( x\right) \right\vert ^{p}d\mu
\left( x\right) \leq c_{p}\underset{s=0}{\overset{N-1}{\sum }}%
\int_{I_{s}\backslash I_{s+1}}M_{s+1}^{p}d\mu \left( x\right)  \label{dir3}
\end{equation}%
\begin{equation}
\leq c_{p}\underset{s=0}{\overset{N-1}{\sum }}M_{s+1}^{p-1}\leq c_{p}N^{%
\left[ p\right] },\text{ \ }0<p\leq 1  \notag
\end{equation}

and
\begin{equation*}
\sum\limits_{n=1}^{\infty }\frac{1}{n^{3-2p}\log ^{2\left[ p\right] }\left(
n+1\right) }\int\limits_{\overline{I}_{N}^{2}}\left\vert S_{n,n}a\left(
x,y\right) \right\vert ^{p}d\mu \left( x,y\right)
\end{equation*}%
\begin{equation*}
\leq \sum\limits_{n=1}^{\infty }\frac{\left\vert \widehat{\Phi }\left(
n\right) \right\vert ^{p}}{n^{3-2p}\log ^{2\left[ p\right] }\left(
n+1\right) }\left( \int_{\overline{I}_{N}}\left\vert
\sum_{j=0}^{N-1}D_{M_{j}}\left( x\right)
\sum_{u=m_{j}-n_{j}}^{m_{j}-1}r_{j}^{u}\left( x\right) \right\vert ^{p}d\mu
\left( x\right) \right) ^{2}
\end{equation*}%
\begin{equation*}
\leq c_{p}N^{2\left[ p\right] }\sum\limits_{n=1}^{\infty }\frac{\left\vert
\widehat{\Phi }\left( n\right) \right\vert ^{p}}{n^{3-2p}\log ^{2\left[ p%
\right] }\left( n+1\right) }.
\end{equation*}

Let $n<M_{N}$. Since $\psi _{n}\left( \tau \right) =1,$ for $\tau \in I_{N}$
we have that
\begin{equation*}
\widehat{\Phi }\left( n\right) =\int\limits_{I_{N}}\Phi \left( \tau \right)
\overline{\psi }_{n}\left( \tau \right) d\mu \left( \tau \right)
=\int\limits_{I_{N}}\left( \int\limits_{I_{N}}a\left( \tau -t,t\right) d\mu
\left( t\right) \right) \overline{\psi }_{n}\left( \tau \right) d\mu \left(
\tau \right)
\end{equation*}%
\begin{equation*}
=\int\limits_{I_{N}^{2}}a\left( s,t\right) d\mu \left( s,t\right) =0.
\end{equation*}%
Hence, we can suppose that $n\geq M_{N}.$ By Hölder inequality we obtain
that
\begin{equation}
N^{2\left[ p\right] }\sum\limits_{n=1}^{\infty }\frac{\left\vert \widehat{%
\Phi }\left( n\right) \right\vert ^{p}}{n^{3-2p}\log ^{2\left[ p\right]
}\left( n+1\right) }\leq \sum\limits_{n=M_{N}}^{\infty }\frac{\left\vert
\widehat{\Phi }\left( n\right) \right\vert ^{p}}{n^{3-2p}}  \label{1p}
\end{equation}%
\begin{equation*}
\leq \left( \sum\limits_{n=M_{N}}^{\infty }\left\vert \widehat{\Phi }\left(
n\right) \right\vert ^{2}\right) ^{p/2}\left( \sum\limits_{n=M_{N}}^{\infty }%
\frac{1}{n^{\left( 3-2p\right) \cdot \left( 2/\left( 2-p\right) \right) }}%
\right) ^{\left( 2-p\right) /2}
\end{equation*}%
\begin{equation*}
\leq \left( \frac{1}{M_{N}^{\left( 2\left( 3-2p\right) /\left( 2-p\right)
-1\right) }}\right) ^{\left( 2-p\right) /2}\left(
\int\limits_{G_{m}}\left\vert \Phi \left( \tau \right) \right\vert ^{2}d\mu
\left( \tau \right) \right) ^{p/2}
\end{equation*}%
\begin{equation*}
\leq \frac{c_{p}}{M_{N}^{\left( 4-3p\right) /2}}\left(
\int\limits_{I_{N}}\left\vert \int\limits_{I_{N}}a\left( \tau -t,t\right)
d\mu \left( t\right) \right\vert ^{2}d\mu \left( \tau \right) \right) ^{p/2}
\end{equation*}%
\begin{equation}
\leq \frac{c_{p}}{M_{N}^{\left( 4-3p\right) /2}}\left\Vert a\right\Vert
_{\infty }^{p}\frac{1}{M_{N}^{p/2}}\frac{1}{M_{N}^{p}}  \notag
\end{equation}%
\begin{equation*}
\leq \frac{c_{p}}{M_{N}^{\left( 4-3p\right) /2}}M_{N}^{2}\frac{1}{%
M_{N}^{3p/2}}=c_{p}<\infty .
\end{equation*}

Let $\left( x,y\right) \in \overline{I}_{N}\times I_{N}$. Then we have that%
\begin{equation*}
S_{n,n}a\left( x,y\right)
\end{equation*}%
\begin{equation*}
=\psi _{n}\left( x\right) \left( \sum_{j=0}^{N-1}D_{M_{j}}\left( x\right)
\sum_{u=m_{j}-n_{j}}^{m_{j}-1}r_{j}^{u}\left( x\right) \right)
\int\limits_{G_{m}^{2}}a\left( s,t\right) \overline{\psi }_{n}\left(
s\right) D_{n}\left( y-t\right) d\mu \left( s,t\right)
\end{equation*}%
\begin{equation*}
=\psi _{n}\left( x\right) \left( \sum_{j=0}^{N-1}D_{M_{j}}\left( x\right)
\sum_{u=m_{j}-n_{j}}^{m_{j}-1}r_{j}^{u}\left( x\right) \right)
\int\limits_{G_{m}}S_{n}^{\left( 2\right) }a\left( s,y\right) \overline{\psi
}_{n}\left( s\right) d\mu \left( s\right)
\end{equation*}%
\begin{equation*}
=\psi _{n}\left( x\right) \left( \sum_{j=0}^{N-1}D_{M_{j}}\left( x\right)
\sum_{u=m_{j}-n_{j}}^{m_{j}-1}r_{j}^{u}\left( x\right) \right) \widehat{S}%
_{n}^{\left( 2\right) }a\left( n,y\right) .
\end{equation*}%
Hence,
\begin{equation*}
\sum\limits_{n=1}^{\infty }\frac{1}{n^{3-2p}\log ^{2\left[ p\right] }\left(
n+1\right) }\int\limits_{\overline{I}_{N}\times I_{N}}\left\vert
S_{n,n}a\left( x,y\right) \right\vert ^{p}d\mu \left( x,y\right)
\end{equation*}%
\begin{equation*}
\leq \sum\limits_{n=1}^{\infty }\frac{1}{n^{3-2p}\log ^{2\left[ p\right]
}\left( n+1\right) }\int\limits_{\overline{I}_{N}\times I_{N}}\left(
\left\vert \sum_{j=0}^{N-1}D_{M_{j}}\left( x\right)
\sum_{u=m_{j}-n_{j}}^{m_{j}-1}r_{j}^{u}\left( x\right) \right\vert
\left\vert \widehat{S}_{n}^{\left( 2\right) }a\left( n,y\right) \right\vert
\right) ^{p}d\mu \left( x,y\right)
\end{equation*}%
\begin{equation*}
\leq \sum\limits_{n=1}^{\infty }\frac{1}{n^{3-2p}\log ^{2\left[ p\right]
}\left( n+1\right) }\int_{\overline{I}_{N}}\left\vert
\sum_{j=0}^{N-1}D_{M_{j}}\left( x\right)
\sum_{u=m_{j}-n_{j}}^{m_{j}-1}r_{j}^{u}\left( x\right) \right\vert ^{p}d\mu
\left( x\right) \cdot \int\limits_{I_{N}}\left\vert \widehat{S}_{n}^{\left(
2\right) }a\left( n,y\right) \right\vert ^{p}d\mu \left( y\right)
\end{equation*}%
\begin{equation*}
\leq c_{p}N^{\left[ p\right] }\sum\limits_{n=1}^{\infty }\frac{1}{%
n^{3-2p}\log ^{2\left[ p\right] }\left( n+1\right) }\int\limits_{I_{N}}\left%
\vert \widehat{S}_{n}^{\left( 2\right) }a\left( n,y\right) \right\vert
^{p}d\mu \left( y\right) .
\end{equation*}

Let $n<M_{N}$. Then, by the definition of atoms we have that
\begin{equation*}
\widehat{S}_{n}^{\left( 2\right) }a\left( n,y\right)
=\int\limits_{G_{m}}\left( \int\limits_{G_{m}}a\left( s,t\right) D_{n}\left(
y-t\right) d\mu \left( t\right) \right) \overline{\psi }_{n}\left( s\right)
d\mu \left( s\right)
\end{equation*}%
\begin{equation*}
=D_{n}\left( y\right) \int\limits_{I_{N}^{2}}a\left( s,t\right) d\mu \left(
s,t\right) =0.
\end{equation*}%
Therefore, we can suppose that $n\geq M_{N}$. Hence
\begin{equation*}
\sum\limits_{n=1}^{\infty }\frac{1}{n^{3-2p}\log ^{2\left[ p\right] }\left(
n+1\right) }\int\limits_{\overline{I}_{N}\times I_{N}}\left\vert
S_{n,n}a\left( x,y\right) \right\vert ^{p}d\mu \left( x,y\right)
\end{equation*}%
\begin{equation*}
\leq c_{p}N^{\left[ p\right] }\sum\limits_{n=M_{N}}^{\infty }\frac{1}{%
n^{3-2p}\log ^{2\left[ p\right] }\left( n+1\right) }\int\limits_{I_{N}}\left%
\vert \widehat{S}_{n}^{\left( 2\right) }a\left( n,y\right) \right\vert
^{p}d\mu \left( y\right) .
\end{equation*}%
Since
\begin{equation*}
\left\Vert \widehat{S}_{n}^{\left( 2\right) }a\left( n,y\right) \right\Vert
_{2}\leq c\left\Vert a\right\Vert _{2},
\end{equation*}%
from Hölder inequality we can write that%
\begin{equation*}
\int\limits_{I_{N}}\left\vert \widehat{S}_{n}^{\left( 2\right) }a\left(
n,y\right) \right\vert ^{p}d\mu \left( y\right) \leq \frac{c_{p}}{M_{N}^{1-p}%
}\left( \int\limits_{I_{N}}\left\vert \widehat{S}_{n}^{\left( 2\right)
}a\left( n,y\right) \right\vert d\mu \left( y\right) \right) ^{p}
\end{equation*}%
\begin{equation*}
=\frac{c_{p}}{M_{N}^{1-p}}\left( \int\limits_{I_{N}}\left\vert
\int\limits_{I_{N}}S_{n}^{\left( 2\right) }a\left( s,y\right) \overline{\psi
}_{n}\left( s\right) d\mu \left( s\right) \right\vert d\mu \left( y\right)
\right) ^{p}
\end{equation*}%
\begin{equation*}
=\frac{c_{p}}{M_{N}^{1-p}}\left( \int\limits_{I_{N}}\left\vert
\int\limits_{I_{N}}\left( \int\limits_{I_{N}}a\left( s,t\right) D_{n}\left(
y-t\right) d\mu \left( t\right) \right) \overline{\psi }_{n}\left( s\right)
d\mu \left( s\right) \right\vert d\mu \left( y\right) \right) ^{p}
\end{equation*}%
\begin{equation*}
\leq \frac{c_{p}}{M_{N}^{1-p}}\left( \int\limits_{I_{N}}\left(
\int\limits_{I_{N}}\left\vert \int\limits_{I_{N}}a\left( s,t\right)
D_{n}\left( y-t\right) d\mu \left( t\right) \right\vert d\mu \left( y\right)
\right) d\mu \left( s\right) \right) ^{p}
\end{equation*}%
\begin{equation*}
\leq \frac{c_{p}}{M_{N}^{1-p}}\left( \frac{1}{M_{N}^{1/2}}%
\int\limits_{I_{N}}\left( \int\limits_{I_{N}}\left\vert
\int\limits_{I_{N}}a\left( s,t\right) D_{n}\left( y-t\right) d\mu \left(
t\right) \right\vert ^{2}d\mu \left( y\right) \right) ^{1/2}d\mu \left(
s\right) \right) ^{p}
\end{equation*}%
\begin{equation*}
\leq \frac{c_{p}}{M_{N}^{1-p}}\left( \frac{1}{M_{N}^{1/2}}%
\int\limits_{I_{N}}\left( \int\limits_{I_{N}}\left\vert a\left( s,t\right)
\right\vert ^{2}d\mu \left( t\right) \right) ^{1/2}d\mu \left( s\right)
\right) ^{p}
\end{equation*}%
\begin{equation*}
\leq \frac{c_{p}}{M_{N}^{1-p}}\left( \frac{\left\Vert a\right\Vert _{\infty }%
}{M_{N}^{1/2}}\frac{1}{M_{N}}\frac{1}{M_{N}^{1/2}}\right) ^{p}\leq \frac{%
c_{p}}{M_{N}^{1-p}}\left( \frac{M_{N}^{2/p}}{M_{N}^{2}}\right) ^{p}\leq
c_{p}M_{N}^{1-p}.
\end{equation*}%
Consequently,
\begin{equation}
\sum\limits_{n=1}^{\infty }\frac{1}{n^{3-2p}\log ^{2\left[ p\right] }\left(
n+1\right) }\int\limits_{\overline{I}_{N}\times I_{N}}\left\vert
S_{n,n}a\left( x,y\right) \right\vert d\mu \left( x,y\right)  \label{2p}
\end{equation}%
\begin{equation*}
\leq c_{p}N^{\left[ p\right] }\sum\limits_{n=M_{N}}^{\infty }\frac{1}{%
n^{3-2p}\log ^{2\left[ p\right] }\left( n+1\right) }M_{N}^{1-p}\leq \left\{
\begin{array}{c}
c_{p}/M_{N}^{1-p},\text{ \ }0<p<1 \\
c,\text{ \ \ \ \ \ }p=1%
\end{array}%
\right. \leq c_{p}<\infty .
\end{equation*}

Analogously, we can prove that
\begin{equation}
\sum\limits_{n=1}^{\infty }\frac{1}{n^{3-2p}\log ^{2\left[ p\right] }\left(
n+1\right) }\int\limits_{I_{N}\times \overline{I}_{N}}\left\vert
S_{n,n}a\left( x,y\right) \right\vert ^{p}d\mu \left( x,y\right) \leq
c_{p}<\infty .  \label{3p}
\end{equation}

Let $\left( x,y\right) \in I_{N}^{2}:=I_{N}\times I_{N}$. Then by the
definition of atoms we can write that%
\begin{equation*}
\int\limits_{I_{N}^{2}}\left\vert S_{n,n}a\left( x,y\right) \right\vert
^{p}d\mu \left( x,y\right)
\end{equation*}%
\begin{equation*}
\leq \frac{1}{M_{N}^{2-p}}\left( \int\limits_{I_{N}^{2}}\left\vert
S_{n,n}a\left( x,y\right) \right\vert ^{2}d\mu \left( x,y\right) \right)
^{p/2}
\end{equation*}%
\begin{equation*}
\leq \frac{1}{M_{N}^{2-p}}\left( \int\limits_{I_{N}^{2}}\left\vert a\left(
x,y\right) \right\vert ^{2}d\mu \left( x,y\right) \right) ^{p/2}
\end{equation*}%
\begin{equation*}
\leq \frac{\left\Vert a\right\Vert _{\infty }^{p}}{M_{N}^{2-p}}\frac{1}{%
M_{N}^{p}}\leq c_{p}\frac{1}{M_{N}^{2-p}}M_{N}^{2}\frac{1}{M_{N}^{p}}\leq
c_{p}<\infty .
\end{equation*}

It follows that
\begin{equation}
\sum\limits_{n=1}^{\infty }\frac{1}{n^{3-2p}\log ^{2\left[ p\right] }\left(
n+1\right) }\int\limits_{I_{N}^{2}}\left\vert S_{n,n}a\left( x,y\right)
\right\vert ^{p}d\mu \left( x,y\right)  \label{4p}
\end{equation}%
\begin{equation*}
\leq c_{p}\sum\limits_{n=1}^{\infty }\frac{1}{n^{3-2p}\log ^{2\left[ p\right]
}\left( n+1\right) }\leq c_{p}<\infty .
\end{equation*}

By combining (\ref{main}-\ref{4p}) we complete the proof of Theorem \ref%
{theorem1}.
\end{proof}

\bigskip

\begin{proof}[\textbf{Proof of Corollary \protect\ref{corollary1}.}]
Suppose that\textbf{\ }
\begin{equation}
\sum\limits_{n=1}^{\infty }\frac{\left\Vert S_{n,n}f\right\Vert _{p}^{p}}{%
n^{3-2p}\log ^{2\left[ p\right] }\left( n+1\right) }\leq c_{p}\left\Vert
f\right\Vert _{H_{p}^{\square }}^{p}.  \label{main1}
\end{equation}

Since%
\begin{equation*}
\left\Vert \sup_{k}\left\vert S_{M_{k},M_{k}}f\right\vert \right\Vert
_{p}\leq c_{p}\left\Vert f\right\Vert _{H_{p}^{\square }}
\end{equation*}%
and
\begin{equation*}
S_{M_{k},M_{k}}S_{n,n}f=\left\{
\begin{array}{c}
S_{M_{k},M_{k}}f,\text{ \ }n>M_{k} \\
S_{n,n}f,\text{ \ \ \ }n\leq M_{k}%
\end{array}%
\right.
\end{equation*}%
we obtain that%
\begin{equation*}
\left\Vert S_{n,n}f\right\Vert _{H_{p}^{\square }}\leq \left\Vert
\sup_{k}\left\vert S_{M_{k},M_{k}}f\right\vert \right\Vert _{p}+\left\Vert
S_{n,n}f\right\Vert _{p}.
\end{equation*}

and%
\begin{equation*}
\sum\limits_{n=1}^{\infty }\frac{\left\Vert S_{n,n}f\right\Vert
_{H_{p}^{\square }}^{p}}{n^{3-2p}\log ^{2\left[ p\right] }\left( n+1\right) }
\end{equation*}%
\begin{equation*}
\leq \sum\limits_{n=1}^{\infty }\frac{\left\Vert S_{n,n}f\right\Vert _{p}^{p}%
}{n^{3-2p}\log ^{2\left[ p\right] }\left( n+1\right) }+\left\Vert
\sup_{k}\left\vert S_{M_{k},M_{k}}f\right\vert \right\Vert
_{p}^{p}\sum\limits_{n=1}^{\infty }\frac{1}{n^{3-2p}\log ^{2\left[ p\right]
}\left( n+1\right) }\leq c_{p}\left\Vert f\right\Vert _{H_{p}^{\square
}}^{p}.
\end{equation*}

This complites the proof of Corollary \ref{corollary1}.
\end{proof}

\begin{proof}[\textbf{Proof} \textbf{of} \textbf{Theorem \protect\ref%
{theorem2}.}]
\textbf{\ }Let $0<p\leq 1$ and $\Phi \left( n\right) $ be any nondecreasing,
nonnegative function, satisfying condition
\begin{equation*}
\underset{n\rightarrow \infty }{\lim }\Phi \left( n\right) =\infty .
\end{equation*}%
For this function $\Phi \left( n\right) ,$ there exists an increasing
sequence of the positive integers $\left\{ \alpha _{k}:\text{ }k\in
\mathbb{N}
\right\} $ such that $\alpha _{0}\geq 2$ and%
\begin{equation}
\sum_{k=0}^{\infty }\frac{\lambda ^{2}}{\Phi ^{p/4}\left( M_{\alpha
_{k}}\right) }<\infty ,  \label{2}
\end{equation}%
where $\lambda =\sup_{n}m_{n}.$

Let \qquad
\begin{equation*}
f_{n,n}\left( x,y\right) =\sum_{\left\{ k;\text{ }\alpha _{k}<n\right\}
}\lambda _{k}a_{k},
\end{equation*}%
where
\begin{equation*}
\lambda _{k}=\frac{\lambda ^{2}}{\Phi ^{1/4}\left( M_{\alpha _{k}}\right) }
\end{equation*}%
and%
\begin{equation*}
a_{k}\left( x,y\right) =\frac{M_{\alpha _{k}}^{2/p-2}}{\lambda ^{2}}\left(
D_{M_{\alpha _{k}+1}}\left( x\right) -D_{M_{\alpha _{k}}}\left( x\right)
\right) \left( D_{M_{\alpha _{k}+1}}\left( y\right) -D_{M_{\alpha
_{k}}}\left( y\right) \right) .
\end{equation*}

It is easy to show that the martingale $\,f=\left( f_{1,1},\text{ }%
f_{2,2},...,\text{ }f_{A,A},\text{ }...\right) \in H_{p}^{\square },$ where $%
0<p\leq 1.$

Indeed, since%
\begin{equation}
S_{M_{A},M_{A}}a_{k}\left( x,y\right) =\left\{
\begin{array}{l}
a_{k}\left( x,y\right) \text{, \ \ }\alpha _{k}<A, \\
0\text{, \qquad\ \ \ \ \ \ }\alpha _{k}\geq A,%
\end{array}%
\right.  \label{4}
\end{equation}%
\begin{equation*}
\text{supp}(a_{k})=I_{\alpha _{k}}^{2},\text{ \ \ \ }\int_{I_{\alpha
_{k}}^{2}}a_{k}d\mu =0
\end{equation*}%
and%
\begin{equation*}
\left\Vert a_{k}\right\Vert _{\infty }\leq \frac{M_{\alpha
_{k}}^{2/p-2}M_{\alpha _{k}+1}^{2}}{\lambda ^{2}}\leq M_{\alpha
_{k}}^{2/p}=\mu (\text{supp }a_{k})^{-1/p}
\end{equation*}%
from Lemma \ref{lemma1} and (\ref{2}) we conclude that $f\in H_{p}^{\square
}\left( G_{m}^{2}\right) ,$ where $0<p\leq 1.$

It is easy to show that%
\begin{equation}
\widehat{f}(i,j)=  \label{5}
\end{equation}%
\begin{equation*}
\left\{
\begin{array}{l}
\frac{M_{\alpha _{k}}^{2/p-2}}{\Phi ^{1/4}\left( M_{\alpha _{k}}\right) }%
,\,\, \\
\text{ if }\left( i,\text{\thinspace }j\right) \in \left\{ M_{\alpha
_{k}},...,\text{ ~}M_{\alpha _{k}+1}-1\right\} \times \left\{ M_{\alpha
_{k}},...,\text{ ~}M_{\alpha _{k}+1}-1\right\} ,\text{ }k=0,1,2... \\
0,\text{ \thinspace } \\
\text{\thinspace \thinspace if \thinspace }\left( \,i,\text{\thinspace }%
j\right) \notin \bigcup\limits_{k=1}^{\infty }\left\{ M_{\alpha _{k}},...,%
\text{ ~}M_{\alpha _{k}+1}-1\right\} \text{ }\times \left\{ M_{\alpha
_{k}},...,\text{ ~}M_{\alpha _{k}+1}-1\right\} .%
\end{array}%
\right.
\end{equation*}

Let\textbf{\ } $M_{\alpha _{k}}<n<M_{\alpha _{k}+1}$. From (\ref{5}) we have
that%
\begin{equation}
S_{n,n}f\left( x,y\right)  \label{13d}
\end{equation}%
\begin{equation*}
=\sum_{i=0}^{M_{\alpha _{k-1}+1}-1}\sum_{j=0}^{M_{\alpha _{k-1}+1}-1-1}%
\widehat{f}(i,j)\psi _{i}\left( x\right) \psi _{j}\left( y\right)
+\sum_{i=M_{\alpha _{k}}}^{n-1}\sum_{j=M_{\alpha _{k}}}^{n-1}\widehat{f}%
(i,j)\psi _{i}\left( x\right) \psi _{j}\left( y\right)
\end{equation*}%
\begin{equation*}
=\sum_{\eta =0}^{k-1}\sum_{i=M_{\alpha _{\eta }}}^{M_{\alpha _{\eta
}+1}-1}\sum_{j=M_{\alpha _{\eta }}}^{M_{\alpha _{\eta }+1}-1}\widehat{f}%
(i,j)\psi _{i}\left( x\right) \psi _{j}\left( y\right) +\sum_{i=M_{\alpha
_{k}}}^{n-1}\sum_{j=M_{\alpha _{k}}}^{n-1}\widehat{f}(i,j)\psi _{i}\left(
x\right) \psi _{j}\left( y\right)
\end{equation*}%
\begin{equation*}
=\sum_{\eta =0}^{k-1}\sum_{i=M_{\alpha _{\eta }}}^{M_{\alpha _{\eta
}+1}-1}\sum_{j=M_{\alpha _{\eta }}}^{M_{\alpha _{\eta }+1}-1}\frac{M_{\alpha
_{\eta }}^{2/p-2}}{\Phi ^{1/4}\left( M_{\alpha _{\eta }}\right) }\psi
_{i}\left( x\right) \psi _{j}\left( y\right) +\sum_{i=M_{\alpha
_{k}}}^{n-1}\sum_{j=M_{\alpha _{k}}}^{n-1}\frac{M_{\alpha _{k}}^{2/p-2}}{%
\Phi ^{1/4}\left( M_{\alpha _{k}}\right) }\psi _{i}\left( x\right) \psi
_{j}\left( y\right)
\end{equation*}%
\begin{equation*}
=\sum_{\eta =0}^{k-1}\frac{M_{\alpha _{\eta }}^{2/p-2}}{\Phi ^{1/4}\left(
M_{\alpha _{\eta }}\right) }\left( D_{M_{\alpha _{\eta }+1}}\left( x\right)
-D_{M_{\alpha _{\eta }}}\left( x\right) \right) \left( D_{M_{\alpha _{\eta
}+1}}\left( y\right) -D_{M_{\alpha _{\eta }}}\left( y\right) \right)
\end{equation*}%
\begin{equation*}
+\frac{M_{\alpha _{k}}^{2/p-2}}{\Phi ^{1/4}\left( M_{\alpha _{k}}\right) }%
\left( D_{n}\left( x\right) -D_{M_{\alpha _{k}}}\left( x\right) \right)
\left( D_{n}\left( y\right) -D_{M_{\alpha _{k}}}\left( y\right) \right)
=I+II.
\end{equation*}

Let $\left( x,y\right) \in \left( G_{m}\backslash I_{1}\right) \times \left(
G_{m}\backslash I_{1}\right) $ and $n\in \mathbb{N}_{n_{0}}$. Since $%
M_{\alpha _{k}}<n<M_{\alpha _{k}+1},$ where $\alpha _{k}\geq 2$ $\left( k\in
\mathbb{N}\right) $, from (\ref{1dn}), (\ref{2dn}) and (\ref{3dn}) we can
write that%
\begin{equation}
D_{M_{s}}\left( z\right) =0,\text{ for }z\in \left( G\backslash I_{1}\right)
\text{ and }s\geq 2  \label{13a0}
\end{equation}%
and%
\begin{equation}
\left\vert II\right\vert =\frac{M_{\alpha _{k}}^{2/p-2}}{\Phi ^{1/4}\left(
M_{\alpha _{k}}\right) }\left\vert D_{n}\left( x\right) D_{n}\left( y\right)
\right\vert  \label{13a}
\end{equation}%
\begin{equation*}
=\frac{M_{\alpha _{k}}^{2/p-2}}{\Phi ^{1/4}\left( M_{\alpha _{k}}\right) }%
\left\vert \psi _{M_{\alpha _{k}}}\left( x\right) \psi _{n-M_{\alpha
_{k}}}\left( x\right) D_{_{1}}\left( x\right) \psi _{M_{\alpha _{k}}}\left(
y\right) \psi _{n-M_{\alpha _{k}}}\left( y\right) D_{_{1}}\left( y\right)
\right\vert =\frac{M_{\alpha _{k}}^{2/p-2}}{\Phi ^{1/4}\left( M_{\alpha
_{k}}\right) }.
\end{equation*}%
By applying (\ref{13a0}) and condition $\alpha _{n}\geq 2$ $\left( n\in
\mathbb{N}\right) $ for $I$ we have that%
\begin{equation}
I=\sum_{\eta =0}^{k-1}\frac{M_{\alpha _{\eta }}^{2/p-2}}{\Phi ^{1/4}\left(
M_{\alpha _{\eta }}\right) }\left( D_{M_{\alpha _{\eta }+1}}\left( x\right)
-D_{M_{\alpha _{\eta }}}\left( x\right) \right) \left( D_{M_{\alpha _{\eta
}+1}}\left( y\right) -D_{M_{\alpha _{\eta }}}\left( y\right) \right) =0.
\label{13b}
\end{equation}%
Hence%
\begin{equation}
\left\Vert S_{n,n}f\left( x,y\right) \right\Vert _{weak-L_{p}}  \label{13}
\end{equation}%
\begin{equation}
\geq \frac{M_{\alpha _{k}}^{2/p-2}}{2\Phi ^{1/4}\left( M_{\alpha
_{k}}\right) }\left( \mu \left\{ \left( x,y\right) \in \left(
G_{m}\backslash I_{1}\right) \times \left( G_{m}\backslash I_{1}\right)
:\left\vert S_{n,n}f\left( x,y\right) \right\vert \geq \frac{M_{\alpha
_{k}}^{2/p-2}}{2\Phi ^{1/4}\left( M_{\alpha _{k}}\right) }\right\} \right)
^{1/p}  \notag
\end{equation}%
\begin{equation*}
\geq \frac{M_{\alpha _{k}}^{2/p-2}}{2\Phi ^{1/4}\left( M_{\alpha
_{k}}\right) }\left\vert \left( G_{m}\backslash I_{1}\right) \times \left(
G_{m}\backslash I_{1}\right) \right\vert ^{1/p}\geq \frac{c_{p}M_{\alpha
_{k}}^{2/p-2}}{\Phi ^{1/4}\left( M_{\alpha _{k}}\right) }.
\end{equation*}

By combining (\ref{1a}) and (\ref{13}) we have that%
\begin{equation}
\underset{n=1}{\overset{M_{\alpha _{k}+1}-1}{\sum }}\frac{\left\Vert
S_{n,n}f\right\Vert _{weak-L_{p}}^{p}\Phi \left( n\right) }{n^{3-2p}}
\label{14}
\end{equation}%
\begin{equation*}
\geq \underset{n=M_{\alpha _{k}}+1}{\overset{M_{\alpha _{k}+1}-1}{\sum }}%
\frac{\left\Vert S_{n,n}f\right\Vert _{weak-L_{p}}^{p}\Phi \left( n\right) }{%
n^{3-2p}}
\end{equation*}%
\begin{equation*}
\geq c_{p}\Phi \left( M_{\alpha _{k}}\right) \underset{\left\{ n:M_{\alpha
_{k}}\leq n\leq M_{\alpha _{k}+1},\text{ }n\in \mathbb{N}_{n_{0}}\right\} }{%
\sum }\frac{\left\Vert S_{n,n}f\right\Vert _{weak-L_{p}}^{p}}{n^{3-2p}}
\end{equation*}%
\begin{equation}
\geq c_{p}\Phi \left( M_{\alpha _{k}}\right) \frac{M_{\alpha _{k}}^{2-2p}}{%
\Phi ^{1/4}\left( M_{\alpha _{k}}\right) M_{\alpha _{k}+1}^{3-2p}}\underset{%
\left\{ n:M_{\alpha _{k}}\leq n\leq M_{\alpha _{k}+1},\text{ }n\in \mathbb{N}%
_{n_{0}}\right\} }{\sum }1  \notag
\end{equation}%
\begin{equation*}
\geq c_{p}\Phi ^{3/4}\left( M_{\alpha _{k}}\right) \rightarrow \infty ,\text{
\qquad when }k\rightarrow \infty .
\end{equation*}

This completes the proof of Theorem \ref{theorem2} when $0<p<1$.

Now, we prove part b) of Theorem \ref{theorem2}. Let\textbf{\ } $M_{\alpha
_{k}}<n<M_{\alpha _{k}+1}$. Since%
\begin{equation}
D_{l+M_{\alpha _{k}}}=D_{M_{\alpha _{k}}}+\psi _{M_{\alpha _{k}}}D_{l},\text{
when }\,\,l<M_{\alpha _{k}},  \label{13e}
\end{equation}

if we apply (\ref{5}) and (\ref{13d}) we obtain that%
\begin{equation}
S_{n,n}f\left( x,y\right)  \label{93}
\end{equation}%
\begin{equation*}
=\frac{\psi _{M_{\alpha _{k}}}\left( x\right) \psi _{M_{\alpha _{k}}}\left(
y\right) D_{n-M_{\alpha _{k}}}\left( x\right) D_{n-M_{\alpha _{k}}}\left(
y\right) }{\Phi ^{1/4}\left( M_{\alpha _{k}}\right) }
\end{equation*}%
\begin{equation*}
+\sum_{\eta =0}^{k-1}\frac{\left( D_{M_{\alpha _{\eta }+1}}\left( x\right)
-D_{M_{\alpha _{\eta }}}\left( x\right) \right) \left( D_{M_{\alpha _{\eta
}+1}}\left( y\right) -D_{M_{\alpha _{\eta }}}\left( y\right) \right) }{\Phi
^{1/4}\left( M_{\alpha _{\eta }}\right) }.
\end{equation*}

Since (see Lukyanenko \cite{luk})%
\begin{equation*}
\frac{1}{4\lambda }v\left( n\right) +\frac{1}{\lambda }v^{\ast }\left(
n\right) +\frac{1}{2\lambda }\leq L_{n}\leq \frac{3}{2}v\left( n\right)
+4v^{\ast }\left( n\right) -1,
\end{equation*}

where $v\left( n\right) $ and $v^{\ast }\left( n\right) $ are defined by (%
\ref{100}) and $\lambda :=\sup_{n\in \mathbb{N}}m_{n},$ according to (\ref%
{1dn}) and (\ref{93}) we have that%
\begin{equation}
\left\Vert S_{n,n}f\left( x,y\right) \right\Vert _{1}\geq \frac{1}{\Phi
^{1/4}\left( M_{\alpha _{k}}\right) }\left\Vert D_{n-M_{\alpha _{k}}}\left(
x\right) D_{n-M_{\alpha _{k}}}\left( y\right) \right\Vert _{1}  \label{6}
\end{equation}%
\begin{equation*}
-\sum_{\eta =0}^{k-1}\frac{\left\Vert \left( D_{M_{\alpha _{\eta }+1}}\left(
x\right) -D_{M_{\alpha _{\eta }}}\left( x\right) \right) \left( D_{M_{\alpha
_{\eta }+1}}\left( y\right) -D_{M_{\alpha _{\eta }}}\left( y\right) \right)
\right\Vert _{1}}{\Phi ^{1/4}\left( M_{\alpha _{\eta }}\right) }
\end{equation*}%
\begin{equation*}
\geq \frac{cL^{2}\left( n-M_{\alpha _{k}}\right) }{\Phi ^{1/4}\left(
M_{\alpha _{k}}\right) }-\sum_{\eta =0}^{\infty }\frac{1}{\Phi ^{1/4}\left(
M_{\alpha _{\eta }}\right) }\geq \frac{cv^{2}\left( n-M_{\alpha _{k}}\right)
}{\Phi ^{1/4}\left( M_{\alpha _{k}}\right) }-C.
\end{equation*}

By combining Lemma \ref{lemma2}, (\ref{6}) and Cauchy-Schwarz inequality we
obtain that%
\begin{equation*}
\underset{l=1}{\overset{M_{\alpha _{k}+1}}{\sum }}\frac{\left\Vert
S_{l,l}f\right\Vert _{1}\Phi \left( l\right) }{l\log ^{2}\left( l+1\right) }%
\geq \underset{l=M_{\alpha _{k}}+1}{\overset{M_{\alpha _{k}+1}}{\sum }}\frac{%
\left\Vert S_{l,l}f_{n,n}\right\Vert _{1}\Phi \left( l\right) }{l\log
^{2}\left( l+1\right) }
\end{equation*}%
\begin{equation*}
\geq \frac{c\Phi \left( M_{\alpha _{k}}\right) }{\alpha _{k}^{2}M_{\alpha
_{k}+1}\Phi ^{1/4}\left( M_{\alpha _{k}}\right) }\underset{l=M_{\alpha
_{k}}+1}{\overset{M_{\alpha _{k}+1}}{\sum }}v^{2}\left( l-M_{\alpha
_{k}}\right) \geq \frac{c\Phi \left( M_{\alpha _{k}}\right) }{\alpha
_{k}^{2}M_{\alpha _{k}+1}\Phi ^{1/4}\left( M_{\alpha _{k}}\right) }\underset{%
l=1}{\overset{M_{\alpha _{k}}}{\sum }}v^{2}\left( l\right)
\end{equation*}%
\begin{equation*}
\geq c\Phi ^{3/4}\left( M_{\alpha _{k}}\right) \left( \frac{1}{\alpha
_{k}M_{\alpha _{k}}}\underset{l=1}{\overset{M_{\alpha _{k}}}{\sum }}v\left(
l\right) \right) ^{2}\geq c\Phi ^{3/4}\left( M_{\alpha _{k}}\right)
\rightarrow \infty ,\text{ when \ }k\rightarrow \infty ,
\end{equation*}

which complete the proof of Theorem \ref{theorem2}.
\end{proof}

\end{document}